\documentclass[12pt,leqno]{amsart}

\usepackage{amsmath,amssymb,amsthm}
\usepackage{cite}
\usepackage[greek,english]{babel}
\usepackage{lmodern}
\usepackage[margin=1.45in,footskip=1.5cm]{geometry}
\usepackage{mathtools}
\usepackage{microtype}
\usepackage[T1]{fontenc}
\usepackage[utf8]{inputenc}

\makeatletter
\@namedef{subjclassname@2020}{\textup{2020} Mathematics Subject Classification}
\makeatother

\def\mathbi#1{\textbf{\em #1}}

\newtheorem{theorem}{Theorem}[section]
\theoremstyle{definition}
\newtheorem{definition}[theorem]{Definition}

\title[Modal probability spaces]{On the foundations of logic and probability, II\\ Modal probability spaces}

\author{Lo\"{i}c Cosyns}
\address{Department of Mathematical Sciences, University of Durham, UK}
\email{l.j.g.cosyns@gmail.com}
\thanks{This paper is an abridged version of the third chapter of the author's thesis \cite{Cosyns26} presented for the degree of Doctor of Philosophy of the University of Durham (UK) and is the second of a series on the foundations of logic and probability. A companion paper on modal sentential calculi has been submitted and is under review; a third paper on formal agnoiology is in preparation; a final paper on iterated modalities is under consideration.}


\begin{document}

\begin{abstract}
A class of models which brings to the fore the inner logic of the calculus of probability instead of assigning probabilities to sentences is introduced and investigated.
\end{abstract}

\maketitle


\section{Introduction}

This study addresses the problem of the efficient integration of logic and probability \cite{Carnap50,Gaifman64,ScottKrauss67}. The adjective `efficient' is to be taken seriously. For logic and probability need not merely be integrated; they need to be integrated in a way which is primarily satisfactory to mathematicians. For if philosophers are always keen to devise new models to integrate logic and probability, few mathematicians venture into the foundations of probability beyond those established by Kolmogorov \cite{Kolmogorov48,Kolmogorov56}, the general consensus being that a revamping of those foundations is, to paraphrase Tarski, neither necessary nor desirable. Accordingly, any viable attempt to integrate logic and probability must leave the axiomatic treatment of probability formulated by Kolmogorov unscathed; in particular, any viable attempt to integrate logic and probability must be carried out using the tools and methods of measure theory. Only if these requirements are met will a mathematician be interested in further investigations into the foundations of probability. Of course, if, incidentally, a problem relating to those foundations is resolved or a neglected chapter of the calculus of probability is revealed, then such investigations will be viewed more favourably.

The standard model for probability is that of a probability space $(X,\Sigma,P)$, where $X$ is a non-empty set, $\Sigma$ is a $\sigma$-field of subsets of $X$, and $P$ is a non-negative, countably additive, normed measure on $\Sigma$. Probability spaces are the dominant model in pure and applied probability. From the standpoint of the foundations of probability, however, probability spaces are not exempt from a certain criticism. A classic point of contention concerns the countable additivity of the probability measure $P$ \cite{DeFinetti75,Kolmogorov48}; another one concerns the representation of every event as a set, which can be viewed as an impoverishment of the formalism \cite{Los55,ScottKrauss67}. More problematic, however, is the existence of non-empty negligible sets, which prevents $P$ from being strictly positive \cite{Kolmogorov48,Los55}. This difficulty is routinely circumvented by forming the probability algebra $(\Sigma/\mathcal{N}_P,\mathbi{P})$ of $(X,\Sigma,P)$, where $\mathcal{N}_P$ is the $\sigma$-ideal of $P$-negligible sets. Factoring out a class of measurable sets, however, although mathematically convenient, it must be acknowledged, does not solve the problem of its existence.

This study addresses that specific issue whilst efficiently integrating logic and probability, which remains the primary objective of this paper. The idea is simple and consists in bringing to the fore the inner logic of the calculus of probability instead of assigning probabilities to sentences. This is discussed in detail in Section 3 after some preliminary considerations in Section 2. Note that this study presupposes some familiarity with the first paper of this series on the foundations of logic and probability \cite{Cosyns27}. In particular, some proofs in the present study heavily rely on theorems established therein. To facilitate cross-referencing, the roman numeral `I' prefixes every theorem of \cite{Cosyns27} called upon in this study (e.g. Theorem I.4.24 refers to Theorem 4.24 in \cite{Cosyns27}).


\section{Preliminaries}

Let $(\mathfrak{A},Z,\mathcal{E},\zeta)$ be a model for modal sentential calculi \cite{Cosyns27}, where $\mathfrak{A}$ is an almost measurable Lindenbaum--Tarski algebra of the language $\mathcal{L}$ of classical sentential calculi, $Z$ is its Stone space, $\mathcal{E}=\mathcal{B}a(Z)/\mathcal{N}$ is the algebra of Baire subsets of $Z$ modulo nowhere dense sets, and $\zeta$ is a Boolean $\sigma$-isomorphism of $\mathfrak{A}$ into $\mathcal{E}$ such that, for every equivalence class $[F]$ of formulas of $\mathcal{L}$, the open-closed representative $\zeta([F])^\ast$ of the image $\zeta([F])$ of $[F]$ is nowhere dense\linebreak in $Z$ (`impossible') if and only if it is the empty set.

The question naturally arises as to what extent $(\mathfrak{A},Z,\mathcal{E},\zeta)$ is compatible with the axiomatic treatment of probability formulated by Kolmogorov \cite{Kolmogorov48,Kolmogorov56}. Since $\mathfrak{A}$ and $\mathcal{E}$ are isomorphic, it is immediately clear that this question can be reduced to that of the existence of a strictly positive, countably additive probability measure ${\bf p}$ on $\mathfrak{A}$, which, it is recalled, satisfies the countable chain condition (ccc) and is weakly countably distributive \cite{Cosyns27}.

Von Neumann \cite{VonNeumann15} famously asked whether these conditions are sufficient for the existence of a strictly positive, countably additive measure on a $\sigma$-complete Boolean algebra; the next theorem shows, in the context of this study, that this is not the case in general.

\begin{theorem}[Talagrand \cite{Talagrand08}]
{\it There exists an almost measurable Lindenbaum--Tarski algebra $\mathfrak{A}$ of the language $\mathcal{L}$ of classical sentential calculi which does not carry a strictly positive, countably additive probability measure {\bf p}.}
\end{theorem}

Theorem 2.1 shows that not every almost measurable Lindenbaum--Tarski algebra of the language of classical sentential calculi carries a strictly positive, countably additive probability measure. It follows that the question of the compatibility of $(\mathfrak{A},Z,\mathcal{E},\zeta)$ with the calculus of probability, which first morphed into that of the existence of a strictly positive, countably additive probability measure ${\bf p}$ on $\mathfrak{A}$, can yet be reframed as the question of the identification of a necessary and sufficient criterion for the existence of ${\bf p}$.

\begin{definition}[Kelley \cite{Kelley59}]
Let $\mathcal{A}\subseteq\mathfrak{A}$ be a non-empty set of equivalence classes of formulas of the language $\mathcal{L}$ of classical sentential calculi. For every indexed family $\mathcal{B}=\{[F_1],[F_2],\dots,[F_n]\}\subseteq\mathcal{A}$ of (not necessarily pairwise disjoint) equivalence classes of formulas of $\mathcal{L}$, the quantity
\begin{gather*}
\textstyle{\#(\mathcal{A})=\inf\,\{m(\mathcal{B})/n(\mathcal{B}):\mathcal{B}\subseteq\mathcal{A}\},}
\end{gather*}
where $m(\mathcal{B})$ is the maximum number of equivalence classes of formulas of $\mathcal{L}$ in $\mathcal{B}$ having non-empty intersection and $n(\mathcal{B})$ is the number of equivalence classes of formulas in $\mathcal{B}$, is called the {\it intersection number of $\mathcal{A}$}.
\end{definition}

\begin{theorem}[Kelley \cite{Kelley59}]
{\it An almost measurable Lindenbaum--Tarski algebra $\mathfrak{A}$ of the language $\mathcal{L}$ of classical sentential calculi carries a strictly positive, countably additive probability measure ${\bf p}$ if and only if $\mathfrak{A}\setminus\{[{\bf 0}]\}$ is the union of countably many non-empty sets of equivalence classes of formulas of $\mathcal{L}$, each of which has a positive intersection number.}
\end{theorem}

More precisely, an almost measurable Lindenbaum--Tarski algebra $\mathfrak{A}$ carries a strictly positive, countably additive probability measure ${\bf p}$ if and only if it carries a strictly positive, finitely additive probability measure ${\bf q}$ if and only if $\mathfrak{A}\setminus\{[{\bf 0}]\}$ is the union of countably many non-empty sets of equivalence classes of formulas of $\mathcal{L}$, each of which has a positive intersection number \cite{Kelley59}.

An almost measurable Lindenbaum--Tarski algebra $\mathfrak{A}$ of the language $\mathcal{L}$ of classical sentential calculi which satisfies Theorem 2.3 is called {\it measurable}. Every measurable Lindenbaum--Tarski algebra is almost measurable, but not conversely. Measurable Lindenbaum--Tarski algebras are denoted by $\mathfrak{A}^\ast$.

An ordered system $(\mathfrak{A}^\ast,{\bf p})$, where $\mathfrak{A}^\ast$ is a measurable Lindenbaum--Tarski algebra and ${\bf p}$ is a strictly positive, countably additive probability measure on $\mathfrak{A}^\ast$, is called a {\it Lindenbaum--Tarski probability algebra of the language $\mathcal{L}$ of classical sentential calculi}.

\begin{theorem}[Stone \cite{Stone37a}]
{\it Let $\mathfrak{A}^\ast$ be a measurable Lindenbaum--Tarski algebra of the language $\mathcal{L}$ of classical sentential calculi, $Z^\ast$ be a set of maximal filters of $\mathfrak{A}^\ast$, and $\gamma$ be a Boolean homomorphism of $\mathfrak{A}^\ast$ into $\mathcal{P}(Z^\ast)$ such that, for every equivalence class $[F]$ of formulas of $\mathcal{L}$, $\gamma([F])=\{\mathbf{F}\in Z^\ast:[F]\in\mathbf{F}\}$. Then the set $Z^\ast$ equipped with the topology for which the $\gamma([F])$'s form an open base is a zero-dimensional compact Hausdorff space.}
\end{theorem}

The zero-dimensional compact Hausdorff space $Z^\ast$ of Theorem 2.4 is called the {\it Stone space of $\mathfrak{A}^\ast$}. Of course, since every measurable Lindenbaum--Tarski algebra is almost measurable, that is to say, ccc and weakly countably distributive \cite{Cosyns27}, $Z^\ast$ is an extremally disconnected compact Hausdorff space in which meagre sets are nowhere dense \cite{Kelley59,Stone37b}. An extremally disconnected compact Hausdorff space is also called a Stonean space.

\begin{theorem}[Fremlin \cite{Fremlin89}]
{\it Every Lindenbaum--Tarski probability algebra $(\mathfrak{A}^\ast,{\bf p})$ of the language $\mathcal{L}$ of classical sentential calculi is isomorphic to the probability algebra $(\mathcal{B}a(Z^\ast)/\mathcal{N},{\bf P})$ of the probability space $(Z^\ast,\mathcal{B}a(Z^\ast),\mathbb{P})$, where $Z^\ast$ is the Stone space of $\mathfrak{A}^\ast$, $\mathcal{B}a(Z^\ast)$ is the $\sigma$-field of Baire subsets of $Z^\ast$, and $\mathbb{P}$ is a strictly positive, countably additive probability measure on $\mathcal{B}a(Z^\ast)$.}
\end{theorem}

The strict positivity of $\mathbb{P}$ in Theorem 2.5 follows from the strict positivity of ${\bf p}$ in the Lindenbaum--Tarski probability algebra $(\mathfrak{A}^\ast,{\bf p})$. Note that the $\sigma$-ideal $\mathcal{N}$ of nowhere dense subsets of $Z^\ast$ is exactly that of negligible subsets of $Z^\ast$. This follows from the existence of a Boolean $\sigma$-epimorphism of $\mathcal{B}a(Z^\ast)$ into $\mathfrak{A}^\ast$\linebreak whose kernel is the $\sigma$-ideal of nowhere dense sets.

In fact, let $\eta$ be a Boolean $\sigma$-epimorphism of $\mathcal{B}a(Z^\ast)$ into $\mathfrak{A^\ast}$ such that $\ker\,\eta=\{B\in\mathcal{B}a(Z^\ast):\eta(B)=[{\bf 0}]\}=\mathcal{N}$, and set $\mathbb{P}(B)={\bf p}(\eta(B))$ for every $B\in\mathcal{B}a(Z^\ast)$. Then $\mathbb{P}(B)=0$ if and only if ${\bf p}(\eta(B))=0$ if and only if $\eta(B)=[{\bf 0}]$ (because ${\bf p}$ is strictly positive) if and only if $B\in\mathcal{N}$, whence every Baire subset of $Z^\ast$ is negligible if and only if it is nowhere dense as claimed.

Theorem 2.5 can be regarded as the preliminary version of a Stone-like representation theorem for Lindenbaum--Tarski probability algebras $(\mathfrak{A^\ast},{\bf p})$. The interaction between the Stone topology $\mathcal{T}_S^\ast$ on $Z^\ast$ and the probability measure $\mathbb{P}$ on $\mathcal{B}a(Z^\ast)$, however, is not entirely clear. In particular, it is not clear in what sense $\mathcal{T}_S^\ast$ and $\mathbb{P}$ are compatible with each other.

\begin{theorem}[Fremlin \cite{Fremlin89}]
{\it Let $Z^\ast$ be the Stone space of a measurable Lindenbaum--Tarski algebra of the language $\mathcal{L}$ of classical sentential calculi, $\mathcal{T}_S^\ast$ be the Stone topology on $Z^\ast$, $\mathcal{B}a(Z^\ast)$ be the $\sigma$-field of Baire subsets of $Z^\ast$, and $\mathbb{P}$ be a strictly positive, countably additive probability measure on $\mathcal{B}a(Z^\ast)$. Then the ordered system $(Z^\ast,\mathcal{T}_S^\ast,\mathcal{B}a(Z^\ast),\mathbb{P})$ is a Radon probability space.}
\end{theorem}

The topological probability space $(Z^\ast,\mathcal{T}_S^\ast,\mathcal{B}a(Z^\ast),\mathbb{P})$ of Theorem 2.6 is a Radon probability space in the following sense:\\
(i) $Z^\ast$ is Hausdorff (distinct points are separated by disjoint open sets);\\
(ii) $\mathcal{T}_S^\ast\subseteq\mathcal{B}a(Z^\ast)$ (every open set is measurable);\\
(iii) $(Z^\ast,\mathcal{B}a(Z^\ast),\mathbb{P})$ is complete (every negligible set is measurable);\\
(iv) $\mathbb{P}(B)=\sup\,\{\mathbb{P}(K):K\subseteq B,~\text{$K$ compact}\}$ for every $B\in\mathcal{B}a(Z^\ast)$ ($\mathbb{P}$ is inner regular with respect to the compact sets).

That $Z^\ast$ is Hausdorff is obvious. And, since every open set in the Stone topology $\mathcal{T}_S^\ast$ symmetrically differs from an open-closed set by its boundary, which is nowhere dense, open sets in $\mathcal{T}_S^\ast$ are certainly measurable. Further, since every Baire subset of $Z^\ast$ is negligible if and only if it is nowhere dense and nowhere dense sets form a $\sigma$-ideal in $Z^\ast$, $(Z^\ast,\mathcal{B}a(Z^\ast),\mathbb{P})$ is complete. Finally,\linebreak $\mathbb{P}$ is inner regular with respect to the compact sets because every Baire subset of $Z^\ast$ is approximable from below by compact sets.

\begin{definition}[Dixmier \cite{Dixmier51}]
Let $(Z^\ast,\mathcal{T}_S^\ast,\mathcal{B}a(Z^\ast),\mathbb{P})$ be the Radon probability space of a Lindenbaum--Tarski probability algebra $(\mathfrak{A^\ast},{\bf p})$ of the language $\mathcal{L}$ of classical sentential calculi. A probability measure on $\mathcal{B}a(Z^\ast)$ is called {\it normal} if every nowhere dense Baire subset of $Z^\ast$ is negligible.

A Stonean space in which the union of the supports of normal probability measures is dense is said to be {\it hyperstonean}.
\end{definition}

\begin{theorem}
{\it Let $(Z^\ast,\mathcal{T}_S^\ast,\mathcal{B}a(Z^\ast),\mathbb{P})$ be the Radon probability space of a Lindenbaum--Tarski probability algebra $(\mathfrak{A^\ast},{\bf p})$ of the language $\mathcal{L}$ of classical sentential calculi. Then $Z^\ast$ is a hyperstonean space.}
\end{theorem}

\begin{proof}
Let $\mathbb{P}$ be a probability measure on $\mathcal{B}a(Z^\ast)$ with support
\begin{gather*}
\textstyle{{\rm supp}(\mathbb{P})=Z^\ast\setminus\bigcup\,\{G\subseteq Z^\ast,~\text{$G$ open}:\mathbb{P}(G)=0\}.}
\end{gather*}
Since every Baire subset of $Z^\ast$ is negligible if and only if it is nowhere dense, $\mathbb{P}$ is normal and ${\rm supp}(\mathbb{P})$ is dense in $Z^\ast$ because nowhere dense open sets are empty and the complement of a generalised union of empty subsets of $Z^\ast$ is $Z^\ast$ itself, which is certainly dense. It follows that the union of the supports of normal probability measures is dense in $Z^\ast$ because every superset of a dense set is dense, whence $Z^\ast$ is hyperstonean as claimed.
\end{proof}

Every hyperstonean space $Z^\ast$ is a Stonean space in which meagre sets are nowhere dense, but not conversely. In fact, let $Z$ be the Stone space of an almost measurable Lindenbaum--Tarski algebra $\mathfrak{A}$. Then there exists a unique disjoint decomposition $Z=Z^\ast\sqcup Z^{\ast\ast}$ such that\\
(i) $Z^\ast$ is hyperstonean: Every meagre set is nowhere dense, and every normal probability measure has dense support;\\
(ii) $Z^{\ast\ast}$ is not hyperstonean: Every meagre set is nowhere dense, but every probability measure has nowhere dense support (whence no probability measure is normal but that which vanishes identically).

For an example of a Stonean space $Z^{\ast\ast}$, see Dixmier \cite{Dixmier51}.

\begin{definition}
Let $(\mathfrak{A}^\ast,{\bf p})$ be a Lindenbaum--Tarski probability algebra of the language $\mathcal{L}$ of classical sentential calculi. The Radon probability space $(Z^\ast,\mathcal{T}_S^\ast,\mathcal{B}a(Z^\ast),\mathbb{P})$ is called the {\it Stone probability space of $(\mathfrak{A}^\ast,{\bf p})$}.
\end{definition}

\begin{theorem}
{\it Every Lindenbaum--Tarski probability algebra $(\mathfrak{A}^\ast,{\bf p})$ of the language $\mathcal{L}$ of classical sentential calculi is isomorphic to the probability algebra $(\mathcal{B}a(Z^\ast)/\mathcal{N},{\bf P})$ of its Stone probability space $(Z^\ast,\mathcal{T}_S^\ast,\mathcal{B}a(Z^\ast),\mathbb{P})$.}
\end{theorem}


\section{Modal probability spaces}

\begin{definition}
A {\it model for the calculus of probability} is an ordered system $(\boldsymbol{\mathfrak{A}},\boldsymbol{Z},\boldsymbol{\mathcal{E}},\zeta)$, where $\boldsymbol{\mathfrak{A}}=(\mathfrak{A}^\ast,{\bf p})$ is a Lindenbaum--Tarski probability algebra of the language $\mathcal{L}$ of classical sentential calculi, $\boldsymbol{Z}=(Z^\ast,\mathcal{T}_S^\ast,\mathcal{B}a(Z^\ast),\mathbb{P})$ is its Stone probability space, $\boldsymbol{\mathcal{E}}=(\mathcal{B}a(Z^\ast)/\mathcal{N},{\bf P})$ is the probability algebra of $\boldsymbol{Z}$, and $\zeta$ is a measure-preserving Boolean $\sigma$-isomorphism of $\boldsymbol{\mathfrak{A}}$ into $\boldsymbol{\mathcal{E}}$ such that, for every equivalence class $[F]$ of formulas of $\mathcal{L}$,
\begin{gather*}
\tag{VIII}
\textstyle{\mathbb{P}(\zeta([F])^\ast)=0~\text{if and only if}~{\rm int}~{\rm cl}~\zeta([F])^\ast=\emptyset,}
\end{gather*}
the probability of the open-closed representative $\zeta([F])^\ast$ of the image $\zeta([F])$ of $[F]$ being $0$ if and only if $\zeta([F])^\ast$ is nowhere dense in $Z^\ast$ (`impossible').
\end{definition}

That Schema (VIII) is, like Schemas (I) and (III) in \cite{Cosyns27}, a theorem of the metalanguage $\mathcal{L}^E$ of mathematical English hardly needs an explanation and follows immediately from the representation of Lindenbaum--Tarski probability algebras, every Baire subset of $Z^\ast$, hence every open-closed subset thereof, being negligible if and only if it is nowhere dense as previously noted.

\begin{theorem}
{\it Let $(\boldsymbol{\mathfrak{A}},\boldsymbol{Z},\boldsymbol{\mathcal{E}},\zeta)$ be a model for the calculus of probability. For every equivalence class $[F]$ of formulas of $\mathcal{L}$,
\begin{gather*}
\tag{IX}
\textstyle{\mathbb{P}(\zeta([F])^\ast)=1~\text{if and only if}~{\rm cl}~{\rm int}~\zeta([F])^\ast=Z^\ast,}
\end{gather*}
the probability of the open-closed representative $\zeta([F])^\ast$ of the image $\zeta([F])$ of $[F]$ being $1$ if and only if $\zeta([F])^\ast$ is non-boundary in $Z^\ast$ $($`necessary'\,$)$.}
\end{theorem}

\begin{proof}
Suppose that $\mathbb{P}(\zeta(-[F])^\ast)=0$ if and only if $\zeta(-[F])^\ast$ is nowhere dense in $Z^\ast$ (`impossible'), that is to say, $\mathbb{P}(\zeta(-[F])^\ast)=0$ if and only if ${\rm int}~{\rm cl}~\zeta(-[F])^\ast=\emptyset$. Since $\mathbb{P}$ is normed to $1$, it follows immediately that
\begin{align*}
\textstyle{1-\mathbb{P}(\zeta(-[F])^\ast})&=\textstyle{1~\text{if and only if}~Z^\ast\setminus {\rm int}~{\rm cl}~\zeta(-[F])^\ast=Z^\ast}\\
\textstyle{\Longleftrightarrow\hspace{0.25cm}\mathbb{P}(Z^\ast\setminus\zeta(-[F])^\ast})&=\textstyle{1~\text{if and only if}~{\rm cl}~Z^\ast\setminus {\rm cl}~\zeta(-[F])^\ast=Z^\ast}\\
\textstyle{\Longleftrightarrow\hspace{0.25cm}\mathbb{P}(Z^\ast\setminus\zeta(-[F])^\ast})&=\textstyle{1~\text{if and only if}~{\rm cl}~{\rm int}~Z^\ast\setminus\zeta(-[F])^\ast=Z^\ast}\\
\textstyle{\Longleftrightarrow\hspace{1.45cm}\mathbb{P}(\zeta([F])^\ast})&=\textstyle{1~\text{if and only if}~{\rm cl}~{\rm int}~\zeta([F])^\ast=Z^\ast.}
\end{align*}
Hence, if a sentence is necessary if and only if its negation is impossible (principle of duality common to all classical modal calculi), then $\zeta([F])^\ast$ has probability $1$ if and only if it is non-boundary in $Z^\ast$ (`necessary') as claimed.
\end{proof}

\begin{theorem}
{\it Let $(\boldsymbol{\mathfrak{A}},\boldsymbol{Z},\boldsymbol{\mathcal{E}},\zeta)$ be a model for the calculus of probability. For every equivalence class $[F]$ of formulas of $\mathcal{L}$,
\begin{gather*}
\tag{X}
\textstyle{\mathbb{P}(\zeta([F])^\ast)>0~\text{if and only if}~{\rm int}~{\rm cl}~\zeta([F])^\ast\not=\emptyset,}
\end{gather*}
the probability of the open-closed representative $\zeta([F])^\ast$ of the image $\zeta([F])$ of\pagebreak\linebreak $[F]$ being $>0$ if and only if $\zeta([F])^\ast$ is somewhere dense in $Z^\ast$ $($`possibly true'\,$)$.}
\end{theorem}

\begin{proof}
Recall that $\mathbb{P}(\zeta([F])^\ast)=0$ if and only if ${\rm int}~{\rm cl}~\zeta([F])^\ast=\emptyset$. Then $\mathbb{P}(\zeta([F])^\ast)\not=0$ if and only if ${\rm int}~{\rm cl}~\zeta([F])^\ast\not=\emptyset$, that is to say, $\mathbb{P}(\zeta([F])^\ast)>0$ if and only if ${\rm int}~{\rm cl}~\zeta([F])^\ast\not=\emptyset$ as required because $\mathbb{P}$ is non-negative.
\end{proof}

\begin{theorem}
{\it Let $(\boldsymbol{\mathfrak{A}},\boldsymbol{Z},\boldsymbol{\mathcal{E}},\zeta)$ be a model for the calculus of probability. For every equivalence class $[F]$ of formulas of $\mathcal{L}$,
\begin{gather*}
\tag{XI}
\textstyle{\mathbb{P}(\zeta([F])^\ast)<1~\text{if and only if}~{\rm cl}~{\rm int}~\zeta([F])^\ast\not=Z^\ast,}
\end{gather*}
the probability of the open-closed representative $\zeta([F])^\ast$ of the image $\zeta([F])$ of $[F]$ being $<1$ if and only if $\zeta([F])^\ast$ is somewhere codense in $Z^\ast$ $($`possibly false'\,$)$.}
\end{theorem}

\begin{proof}
Recall that $\mathbb{P}(\zeta([F])^\ast)=1$ if and only if ${\rm cl}~{\rm int}~\zeta([F])^\ast=Z^\ast$. Then $\mathbb{P}(\zeta([F])^\ast)\not=1$ if and only if ${\rm cl}~{\rm int}~\zeta([F])^\ast\not=Z^\ast$, that is to say, $\mathbb{P}(\zeta([F])^\ast)<1$ if and only if ${\rm cl}~{\rm int}~\zeta([F])^\ast\not=Z^\ast$ as required because $\mathbb{P}$ is normed to $1$.
\end{proof}

Schemas (VIII)--(XI) define the probability of $\zeta([F])^\ast$ in terms of its mode of truth. Since Schemas (III)--(VI) in \cite{Cosyns27} define the mode of truth of $\zeta([F])^\ast$ in terms of sets, $\mathbb{P}(\zeta([F])^\ast)$, as a sentence of the metalanguage $\mathcal{L}^E$ of mathematical English, can always be decomposed into a statement of probability and a statement of modal logic by combining Schemas (VIII)--(XI) and (III)--(VI). For example, combining Schema (VIII) and (III), $\mathbb{P}(\zeta([F])^\ast)=0$ if and only if\linebreak ${\rm int}~{\rm cl}~\zeta([F])^\ast=\emptyset$ if and only if $\zeta([F])^\ast=\emptyset$. Of course, in the ordinary conduct of his affairs, a probabilist will simply say that $\mathbb{P}(\emptyset)=0$.

\begin{definition}
Let $(\boldsymbol{\mathfrak{A}},\boldsymbol{Z},\boldsymbol{\mathcal{E}},\zeta)$ be a model for the calculus of probability. A {\it semantic bridge} is a finite sequence of logical equivalences, formulated in the metalanguage $\mathcal{L}^E$ of mathematical English, combining a statement of probability and a statement of modal logic.

A {\it semantic decomposition} is a finite sequence of semantic bridges.
\end{definition}

\begin{definition}
Let $(\boldsymbol{\mathfrak{A}},\boldsymbol{Z},\boldsymbol{\mathcal{E}},\zeta)$ be a model for the calculus of probability. An {\it event} is an element of the $\sigma$-field $\mathcal{B}a(Z^\ast)$ of Baire subsets of $Z^\ast$.

In particular, for every equivalence class $[F]$ of formulas of the language $\mathcal{L}$ of classical sentential calculi, the open-closed representative $\zeta([F])^\ast$ of the image $\zeta([F])$ of $[F]$ is called the {\it event corresponding to $[F]$}.
\end{definition}

\begin{theorem}
{\it Let $(\boldsymbol{\mathfrak{A}},\boldsymbol{Z},\boldsymbol{\mathcal{E}},\zeta)$ be a model for the calculus of probability. Then there exists a semantic decomposition such that
\begin{gather*}
\textstyle{\mathbb{P}(\zeta([F])^\ast)\geq 0~\text{for every}~\zeta([F])^\ast\in\mathcal{B}a(Z^\ast),}
\end{gather*}
the probability of every event $\zeta([F])^\ast$ corresponding to an equivalence class $[F]$ of formulas of $\mathcal{L}$ being non-negative.}
\end{theorem}

\begin{proof}
Clearly, $\mathbb{P}(\zeta([F])^\ast)=0$ if and only if ${\rm int}~{\rm cl}~\zeta([F])^\ast=\emptyset$ if and only if $\zeta([F])^\ast=\emptyset$ by Schemas (VIII) and (III); and $\mathbb{P}(\zeta([F])^\ast)>0$ if and only if ${\rm int}~{\rm cl}~\zeta([F])^\ast\not=\emptyset$ if and only if $\zeta([F])^\ast\not=\emptyset$ by Schemas (X) and (V). Hence, there exist semantic bridges such that $\mathbb{P}(\zeta([F])^\ast)=0$ if and only if $\zeta([F])^\ast=\emptyset$ and $\mathbb{P}(\zeta([F])^\ast)>0$ if and only if $\zeta([F])^\ast\not=\emptyset$, whence there exists a semantic decomposition such that $\mathbb{P}(\zeta([F])^\ast)\geq 0$ for every $\zeta([F])^\ast\in\mathcal{B}a(Z^\ast)$.
\end{proof}

\begin{theorem}
{\it Let $(\boldsymbol{\mathfrak{A}},\boldsymbol{Z},\boldsymbol{\mathcal{E}},\zeta)$ be a model for the calculus of probability. Then there exists a semantic decomposition such that
\begin{gather*}
\textstyle{\mathbb{P}(Z^\ast)=1,}
\end{gather*}
the probability of the necessary event $($or sure event\,$)$ being $1$.}
\end{theorem}

\begin{proof}
Clearly, $\mathbb{P}(\zeta([F])^\ast)=1$ if and only if ${\rm cl}~{\rm int}~\zeta([F])^\ast=Z^\ast$ if and only if $\zeta([F])^\ast=Z^\ast$ by Schemas (IX) and (IV), whence there exists a semantic bridge, hence a semantic decomposition, such that $\mathbb{P}(Z^\ast)=1$ as claimed.
\end{proof}

\begin{theorem}
{\it Let $(\boldsymbol{\mathfrak{A}},\boldsymbol{Z},\boldsymbol{\mathcal{E}},\zeta)$ be a model for the calculus of probability. Then there exists a semantic decomposition such that
\begin{gather*}
\textstyle{\mathbb{P}(\zeta(\bigvee_{n\in\mathbb{N}}[F_n])^\ast)=\sum_{n=0}^{\infty}\mathbb{P}(\zeta([F_n])^\ast),}
\end{gather*}
the probability of the event $\zeta(\bigvee_{n\in\mathbb{N}}[F_n])^\ast$ corresponding to the join of an indexed family $\{[F_n]:n\in\mathbb{N}\}$ of pairwise disjoint equivalence classes of formulas of $\mathcal{L}$ being the sum of the probabilities of each event $\zeta([F_0])^\ast,\zeta([F_1])^\ast,\zeta([F_2])^\ast,\dots$ corresponding to each equivalence class $[F_0],[F_1],[F_2],\dots$ of formulas.}
\end{theorem}

\begin{proof}
Suppose that $\mathbb{P}(\zeta(\bigvee_{n\in\mathbb{N}}[F_n])^\ast)=0$. It follows immediately from Schemas (VIII) and (III) plus Theorem I.4.24 that $\mathbb{P}(\zeta(\bigvee_{n\in\mathbb{N}}[F_n])^\ast)=0$ if and only if ${\rm int}~{\rm cl}~\zeta(\bigvee_{n\in\mathbb{N}}[F_n])^\ast=\emptyset$ if and only if ${\rm int}~{\rm cl}~\zeta([F_n])^\ast=\emptyset$ for every $n$ if and only if $\zeta([F_n])^\ast=\emptyset$ for every $n$ if and only if ${\rm int}~{\rm cl}~\zeta([F_n])^\ast=\emptyset$ for every $n$ if and only if $\mathbb{P}(\zeta([F_n])^\ast)=0$ for every $n$ if and only if $\sum_{n=0}^\infty\mathbb{P}(\zeta([F_n])^\ast)=0$ because $\mathbb{P}$ is a non-negative measure. Hence, there exist semantic bridges such that $\mathbb{P}(\zeta(\bigvee_{n\in\mathbb{N}}[F_n])^\ast)=\sum_{n=0}^\infty\mathbb{P}(\zeta([F_n])^\ast)$ whenever $\mathbb{P}(\zeta(\bigvee_{n\in\mathbb{N}}[F_n])^\ast)=0$.

Suppose now that $\mathbb{P}(\zeta(\bigvee_{n\in\mathbb{N}}[F_n])^\ast)>0$. Likewise, it follows from Schemas (X) and (V) plus Theorem I.4.20 that $\mathbb{P}(\zeta(\bigvee_{n\in\mathbb{N}}[F_n])^\ast)>0$ if and only if ${\rm int}~{\rm cl}~\zeta(\bigvee_{n\in\mathbb{N}}[F_n])^\ast\not=\emptyset$ if and only if ${\rm int}~{\rm cl}~\zeta([F_n])^\ast\not=\emptyset$ for some $n$ if and only if $\zeta([F_n])^\ast\not=\emptyset$ for some $n$ if and only if ${\rm int}~{\rm cl}~\zeta([F_n])^\ast\not=\emptyset$ for some $n$ if and only if $\mathbb{P}(\zeta([F_n])^\ast)>0$ for some $n$ if and only if $\sum_{n=0}^\infty\mathbb{P}(\zeta([F_n])^\ast)>0$ because $\sum_{n=0}^\infty\mathbb{P}(\zeta([F_n])^\ast)=0$ otherwise. Hence, since $\mathbb{P}(\zeta(\bigvee_{n\in\mathbb{N}}[F_n])^\ast)=\mathbb{P}(\bigvee_{n\in\mathbb{N}}\zeta([F_n])^\ast)=\mathbb{P}(\bigcup_{n\in\mathbb{N}}\zeta([F_n])^\ast)=\sum_{n=0}^{\infty}\mathbb{P}(\zeta([F_n])^\ast)$ because $\zeta$ is a Boolean $\sigma$-isomorphism, $\bigvee_{n\in\mathbb{N}}\zeta([F_n])^\ast$ and $\bigcup_{n\in\mathbb{N}}\zeta([F_n])^\ast$ belong to the same equivalence class of measurable sets, and $\mathbb{P}$ is countably additive, there exist semantic bridges such that $\mathbb{P}(\zeta(\bigvee_{n\in\mathbb{N}}[F_n])^\ast)=\sum_{n=0}^{\infty}\mathbb{P}(\zeta([F_n])^\ast)$ whenever $\mathbb{P}(\zeta(\bigvee_{n\in\mathbb{N}}$\linebreak $[F_n])^\ast)>0$. A glance at Definition 3.5 completes the proof.
\end{proof}

Theorems 3.7, 3.8, and 3.9 establish the existence of a semantic decomposition for the axioms of probability as formulated, in a modernised and logicised form, by Kolmogorov \cite{Kolmogorov56}. A detailed examination of the semantic bridges was necessary to show how logic and probability are integrated in practice. Note that these theorems do not show that $\mathbb{P}$ is a non-negative, countably additive, normed measure, a definition which held {\it a priori}, but that the logic introduced in \cite{Cosyns27} is consistent with the calculus of probability.

\begin{theorem}
{\it Let $(\boldsymbol{\mathfrak{A}},\boldsymbol{Z},\boldsymbol{\mathcal{E}},\zeta)$ be a model for the calculus of probability. Then there exists a semantic decomposition such that
\begin{gather*}
\textstyle{\mathbb{P}(\zeta(-[F])^\ast)=1-\mathbb{P}(\zeta([F])^\ast),}
\end{gather*}
the probability of the event $\zeta(-[F])^\ast$ corresponding to the complement $-[F]$ of an equivalence class $[F]$ of formulas of $\mathcal{L}$ being the difference to $1$ of the probability of the event $\zeta([F])^\ast$ corresponding to $[F]$.}
\end{theorem}

\begin{proof}
Since $\zeta([F]\vee -[F])^\ast=Z^\ast$, there exists a semantic bridge such that $\mathbb{P}(\zeta([F]\vee -[F])^\ast)=1$ by Theorem 3.8; and since $[F]$ and $-[F]$ are pairwise disjoint, that is to say, $[F]\wedge -[F]=[{\bf 0}]$, there exist semantic bridges such that $\mathbb{P}(\zeta([F]\vee -[F])^\ast)=\mathbb{P}(\zeta([F])^\ast)+\mathbb{P}(\zeta(-[F])^\ast)$ by Theorem 3.9. Hence, there exists a semantic decomposition such that $\mathbb{P}(\zeta([F]\vee -[F])^\ast)=\mathbb{P}(\zeta([F])^\ast)+\mathbb{P}(\zeta(-[F])^\ast)=1$ from which the conclusion follows immediately.
\end{proof}

\begin{theorem}
{\it Let $(\boldsymbol{\mathfrak{A}},\boldsymbol{Z},\boldsymbol{\mathcal{E}},\zeta)$ be a model for the calculus of probability. Then there exists a semantic decomposition such that
\begin{gather*}
\textstyle{\mathbb{P}(\zeta([F_1])^\ast)\leq\mathbb{P}(\zeta([F_2])^\ast)~\text{whenever}~\zeta([F_1])^\ast\subseteq\zeta([F_2])^\ast,}
\end{gather*}
the probability of the event $\zeta([F_1])^\ast$ corresponding to an equivalence class $[F_1]$ of formulas of $\mathcal{L}$ being at most that of the event $\zeta([F_2])^\ast$ corresponding to an equivalence class $[F_2]$ of formulas whenever the event $\zeta([F_1])^\ast$ corresponding to $[F_1]$ is a subset of the event $\zeta([F_2])^\ast$ corresponding to $[F_2]$.}
\end{theorem}

\begin{proof}
Since $\zeta([F_1])^\ast\subseteq\zeta([F_2])^\ast$, there exist an event $\zeta([F_3])^\ast\subseteq\zeta([F_2])^\ast$ such that $\zeta([F_1])^\ast\cup\zeta([F_3])^\ast=\zeta([F_2])^\ast$, with $\zeta([F_1])^\ast\cap\zeta([F_3])^\ast=\emptyset$. It follows that $\zeta([F_1]\vee [F_3])^\ast=\zeta([F_2])^\ast$, with $[F_1]\wedge [F_3]=[{\bf 0}]$, whence there exist semantic bridges such that $\mathbb{P}(\zeta([F_1]\vee [F_3])^\ast)=\mathbb{P}(\zeta([F_1])^\ast)+\mathbb{P}(\zeta([F_3])^\ast=\mathbb{P}(\zeta([F_2])^\ast)$ by Theorem 3.9. Since there exist semantic bridges such that $\mathbb{P}(\zeta([F_3])^\ast)\geq 0$ by Theorem 3.7, there exists a semantic decomposition such that $\mathbb{P}(\zeta([F_1])^\ast)\leq\mathbb{P}(\zeta([F_2])^\ast)$ whenever $\zeta([F_1])^\ast\subseteq\zeta([F_2])^\ast$ as claimed.
\end{proof}

\begin{theorem}
{\it Let $(\boldsymbol{\mathfrak{A}},\boldsymbol{Z},\boldsymbol{\mathcal{E}},\zeta)$ be a model for the calculus of probability. Then there exists a semantic decomposition such that
\begin{gather*}
\textstyle{\mathbb{P}(\zeta([F_1]\vee [F_2])^\ast)=\mathbb{P}(\zeta([F_1])^\ast)+\mathbb{P}(\zeta([F_2])^\ast)-\mathbb{P}(\zeta([F_1]\wedge [F_2])^\ast),}
\end{gather*}
the probability of the event $\zeta([F_1]\vee [F_2])^\ast$ corresponding to the join of two arbitrary equivalence classes $[F_1],[F_2]$ of formulas of $\mathcal{L}$ being the sum of the probabilities of each event $\zeta([F_1])^\ast,\zeta([F_2])^\ast$ corresponding to each equivalence class\linebreak $[F_1],[F_2]$ of formulas minus the probability of the event $\zeta([F_1]\wedge [F_2])^\ast$ corresponding to the meet of $[F_1],[F_2]$.}
\end{theorem}

\begin{proof}
Set $[F_1]\vee [F_2]=([F_1]\wedge -([F_1]\wedge [F_2]))\vee ([F_1]\wedge [F_2])\vee ([F_2]\wedge -([F_1]\wedge [F_2]))$. Since $([F_1]\wedge -([F_1]\wedge [F_2]))$, $([F_1]\wedge [F_2])$, and $([F_2]\wedge -([F_1]\wedge [F_2]))$ are pairwise disjoint, by Theorem 3.9, there certainly exist semantic bridges such that $\mathbb{P}(\zeta([F_1]\vee [F_2])^\ast)=\mathbb{P}(\zeta([F_1]\wedge -([F_1]\wedge [F_2]))^\ast)+\mathbb{P}(\zeta([F_1]\wedge [F_2])^\ast)+\mathbb{P}(\zeta([F_2]\wedge -([F_1]\wedge [F_2]))^\ast)$; and, since $[F_1]\wedge [F_2]\leq [F_1]$ and $[F_1]\wedge [F_2]\leq [F_2]$, still by Theorem 3.9, there certainly exist semantic bridges such that $\mathbb{P}(\zeta([F_1])^\ast)-\mathbb{P}(\zeta([F_1]\wedge [F_2])^\ast)+\mathbb{P}(\zeta([F_1]\wedge [F_2])^\ast)+\mathbb{P}(\zeta([F_2])^\ast)-\mathbb{P}(\zeta([F_1]\wedge [F_2])^\ast)$. Hence, there exists a semantic decomposition such that $\mathbb{P}(\zeta([F_1]\vee [F_2])^\ast)=\linebreak\mathbb{P}(\zeta([F_1])^\ast)+\mathbb{P}(\zeta([F_2])^\ast)-\mathbb{P}(\zeta([F_1]\wedge [F_2])^\ast)$ as claimed.
\end{proof}

Theorems 3.10, 3.11, and 3.12 establish the existence of a semantic decomposition for the complement rule, monotonicity, and the addition rule. Other properties having a semantic decomposition can be derived by mathematical induction (e.g. Poincar\'{e}'s identity or Boole's inequality). Since these properties arise as corollaries of corollaries, they are not reproduced here.

Note that Theorem 3.11 can be framed as follows:
\begin{gather*}
\textstyle{\mathbb{P}(\zeta([F_1])^\ast)\leq\mathbb{P}(\zeta([F_2])^\ast)~\text{whenever}~{\rm cl}~{\rm int}~\zeta([F_1]\rightarrow [F_2])^\ast=Z^\ast,}
\end{gather*}
the probability of the event $\zeta([F_1])^\ast$ corresponding to an equivalence class $[F_1]$ of formulas of $\mathcal{L}$ being at most that of the event $\zeta([F_2])^\ast$ corresponding to an equivalence class $[F_2]$ of formulas of $\mathcal{L}$ whenever the event $\zeta([F_1]\rightarrow [F_2])^\ast$ corresponding to the relative pseudo-complement $[F_1]\rightarrow [F_2]$ is non-boundary in $Z^\ast$. In the context of probability theory, however, the operation of set inclusion is usually preferred to that of relative pseudo-complementation.

\begin{theorem}
{\it Let $(\boldsymbol{\mathfrak{A}},\boldsymbol{Z},\boldsymbol{\mathcal{E}},\zeta)$ be a model for the calculus of probability. Then there exists a semantic decomposition such that
\begin{gather*}
\textstyle{\mathbb{P}(\zeta(\bigwedge_{1\leq k\leq n}[F_{i_k}])^\ast)=\prod_{k=1}^n\mathbb{P}(\zeta([F_{i_k}])^\ast),}
\end{gather*}
the probability of the event $\zeta(\bigwedge_{1\leq k\leq n}[F_{i_k}])^\ast$ corresponding to the meet of an indexed family $\{[F_i]:i\in I\}$ of mutually independent equivalence classes of formulas of $\mathcal{L}$ being the product, for every finite set $\{i_1,i_2,\ldots,i_n\}\subseteq I$, of the probabilities of each event $\zeta([F_{i_1}])^\ast,\zeta([F_{i_2}])^\ast,\ldots,\zeta([F_{i_n}])^\ast$ corresponding to each equivalence class $[F_{i_1}],[F_{i_2}],\dots,[F_{i_n}]$ of formulas.}
\end{theorem}

\begin{proof}
Suppose that $\mathbb{P}(\zeta(\bigwedge_{1\leq k\leq n}[F_{i_k}])^\ast)=1$. It follows immediately from Schemas (IX) and (IV) plus Theorem I.4.14 that, for every indexed set $\{i_1,i_2,\ldots,i_n\}\subseteq I$,\pagebreak\linebreak $\mathbb{P}(\zeta(\bigwedge_{1\leq k\leq n}[F_{i_k}])^\ast)=1$ if and only if ${\rm cl}~{\rm int}~\zeta(\bigwedge_{1\leq k\leq n}[F_{i_k}])^\ast=Z^\ast$ if and only if ${\rm cl}~{\rm int}~\zeta([F_{i_k}])^\ast=Z^\ast$ for every $k$ if and only if $\zeta([F_{i_k}])^\ast=Z^\ast$ for every $k$ if and only if ${\rm cl}~{\rm int}~\zeta([F_{i_k}])^\ast=Z^\ast$ for every $k$ if and only if $\mathbb{P}(\zeta([F_{i_k}])^\ast)=1$ for every $k$ if and only if $\prod_{k=1}^n\mathbb{P}(\zeta([F_{i_k}])^\ast)=1$ because $\mathbb{P}$, as a probability measure, is normed to $1$. Hence, there exist semantic bridges such that $\mathbb{P}(\zeta(\bigwedge_{1\leq k\leq n}[F_{i_k}])^\ast)=\prod_{k=1}^n\mathbb{P}(\zeta([F_{i_k}])^\ast)$ whenever $\mathbb{P}(\zeta(\bigwedge_{1\leq k\leq n}[F_{i_k}])^\ast)=1$.

Suppose now that $\mathbb{P}(\zeta(\bigwedge_{1\leq k\leq n}[F_{i_k}])^\ast)<1$. Likewise, it follows from Schemas (XI) and (VI) plus the contradual of Theorem I.4.20 that, for every indexed set $\{i_1,i_2,\ldots,i_n\}\subseteq I$, $\mathbb{P}(\zeta(\bigwedge_{1\leq k\leq n}[F_{i_k}])^\ast)<1$ if and only if ${\rm cl}~{\rm int}~\zeta(\bigwedge_{1\leq k\leq n}[F_{i_k}])^\ast\not=Z^\ast$ if and only if ${\rm cl}~{\rm int}~\zeta([F_{i_k}])^\ast\not=Z^\ast$ for some $k$ if and only if $\zeta([F_{i_k}])^\ast\not=Z^\ast$ for some $k$ if and only if ${\rm cl}~{\rm int}~\zeta([F_{i_k}])^\ast\not=Z^\ast$ for some $k$ if and only if $\mathbb{P}(\zeta([F_{i_k}])^\ast)<1$ for some $k$ if and only if $\prod_{k=1}^n\mathbb{P}(\zeta([F_{i_k}])^\ast)<1$ because $\prod_{k=1}^n\mathbb{P}(\zeta([F_{i_k}])^\ast)=1$ otherwise. Hence, since $\mathbb{P}(\zeta(\bigwedge_{1\leq k\leq n}[F_{i_k}])^\ast)={\bf p}(\bigwedge_{1\leq k\leq n}[F_{i_k}])=\prod_{k=1}^n{\bf p}([F_{i_k}])=\prod_{k=1}^n\mathbb{P}(\zeta([F_{i_k}])^\ast)$ because $\zeta$ is a measure-preserving Boolean $\sigma$-isomorphism and, by definition, $\{[F_i]:i\in I\}$ is an indexed family of mutually independent equivalence classes of formulas of $\mathcal{L}$, there exist semantic bridges such that $\mathbb{P}(\zeta(\bigwedge_{1\leq k\leq n}[F_{i_k}])^\ast)=\prod_{k=1}^n\mathbb{P}(\zeta([F_{i_k}])^\ast)$ whenever $\mathbb{P}(\zeta(\bigwedge_{1\leq k\leq n}[F_{i_k}])^\ast)<1$.\linebreak A glance at Definition 3.5 completes the proof.
\end{proof}

\begin{theorem}
{\it Let $(\boldsymbol{\mathfrak{A}},\boldsymbol{Z},\boldsymbol{\mathcal{E}},\zeta)$ be a model for the calculus of probability. Then there exists a semantic decomposition such that
\begin{gather*}
\textstyle{\mathbb{P}(\zeta([F_1]\wedge [F_2])^\ast)=\mathbb{P}(\zeta([F_1])^\ast)\times\mathbb{P}(\zeta([F_2])^\ast),}
\end{gather*}
the probability of the event $\zeta([F_1]\wedge [F_2])^\ast$ corresponding to the meet of two independent equivalence classes $[F_1],[F_2]$ of formulas of $\mathcal{L}$ being the product of the probabilities of each event $\zeta([F_1])^\ast,\zeta([F_2])^\ast$ corresponding to each equivalence class $[F_1],[F_2]$ of formulas.}
\end{theorem}

\begin{proof}
As an immediate consequence of Theorem 3.13.
\end{proof}

Theorems 3.13 and 3.14 establish the existence of a semantic decomposition for mutual independence and independence. Like for the axioms of probability, a detailed examination of the semantic bridges was necessary to show how logic and probability are integrated in practice.

\begin{definition}
Let $(\boldsymbol{\mathfrak{A}},\boldsymbol{Z},\boldsymbol{\mathcal{E}},\zeta)$ be a model for the calculus of probability and $\zeta([F_1])^\ast,\zeta([F_2])^\ast$ be two events corresponding to two equivalence classes $[F_1],[F_2]$ of formulas of $\mathcal{L}$ such that $\mathbb{P}(\zeta([F_2])^\ast)>0$. Then
\begin{gather*}
\textstyle{\mathbb{P}_{\zeta([F_2])^\ast}(\zeta([F_1])^\ast)=\frac{\raisebox{3pt}{$\mathbb{P}(\zeta([F_1]\wedge [F_2])^\ast)$}}{\raisebox{-3pt}{$\mathbb{P}(\zeta([F_2])^\ast)$}}}
\end{gather*}
is called the {\it conditional probability of the event $\zeta([F_1])^\ast$ corresponding to $[F_1]$ given the event $\zeta([F_2])^\ast$ corresponding to $[F_2]$}.
\end{definition}

\begin{theorem}
{\it Let $(\boldsymbol{\mathfrak{A}},\boldsymbol{Z},\boldsymbol{\mathcal{E}},\zeta)$ be a model for the calculus of probability. Then there exists a semantic decomposition such that
\begin{gather*}
\textstyle{\mathbb{P}(\zeta([F_1]\wedge [F_2])^\ast)=\mathbb{P}_{\zeta([F_2])^\ast}(\zeta([F_1])^\ast)\times\mathbb{P}(\zeta([F_2])^\ast),}
\end{gather*}
the probability of the event $\zeta([F_1]\wedge [F_2])^\ast$ corresponding to the meet of two not necessarily independent equivalence classes $[F_1],[F_2]$ of formulas of $\mathcal{L}$, provided that the probability of $\zeta([F_2])^\ast$ is positive, being the probability of the event $\zeta([F_1])^\ast$ corresponding to $[F_1]$ given the event $\zeta([F_2])^\ast$ corresponding to $[F_2]$ times the probability of the event $\zeta([F_2])^\ast$ corresponding to $[F_2]$.}
\end{theorem}

\begin{proof}
Clearly, if $[F_1],[F_2]$ are independent, then $\mathbb{P}_{\zeta([F_2])^\ast}(\zeta([F_1])^\ast)=\mathbb{P}(\zeta([F_1])^\ast)$, whence there exists a semantic decomposition such that $\mathbb{P}(\zeta([F_1]\wedge [F_2])^\ast)=\mathbb{P}_{\zeta([F_2])^\ast}(\zeta([F_1])^\ast)\times\mathbb{P}(\zeta([F_2])^\ast)$ by Theorem 3.14.

Suppose that $[F_1],[F_2]$ are dependent and that $\mathbb{P}(\zeta([F_1]\wedge [F_2])^\ast)=1$. It follows immediately from Schemas (IX) and (IV) plus Theorem I.4.14 that $\mathbb{P}(\zeta([F_1]\wedge [F_2])^\ast)=1$ if and only if ${\rm cl}~{\rm int}~\zeta([F_1]\wedge [F_2])^\ast=Z^\ast$ if and only if ${\rm cl}_{\zeta([F_2])^\ast}~{\rm int}_{\zeta([F_2])^\ast}~\zeta([F_1])^\ast=Z^\ast\cap\zeta([F_2])^\ast$ and ${\rm cl}~{\rm int}~\zeta([F_2])^\ast=Z^\ast$ if and only if $\zeta([F_1])^\ast=Z^\ast\cap\,\zeta([F_2])^\ast$ and $\zeta([F_2])^\ast=Z^\ast$ if and only if ${\rm cl}_{\zeta([F_2])^\ast}~{\rm int}_{\zeta([F_2])^\ast}~\zeta([F_1])^\ast\linebreak=Z^\ast\cap\zeta([F_2])^\ast$ and ${\rm cl}~{\rm int}~\zeta([F_2])^\ast=Z^\ast$ if and only if $\mathbb{P}_{\zeta([F_2])^\ast}(\zeta([F_1])^\ast)=1$ and $\mathbb{P}(\zeta([F_2])^\ast)=1$ if and only if $\mathbb{P}_{\zeta([F_2])^\ast}(\zeta([F_1])^\ast)\times\mathbb{P}(\zeta([F_2])^\ast)=1$ because $\mathbb{P}$ is normed to $1$. Hence, there exist semantic bridges such that $\mathbb{P}(\zeta([F_1]\wedge [F_2])^\ast)=\mathbb{P}_{\zeta([F_2])^\ast}(\zeta([F_1])^\ast)\times\mathbb{P}(\zeta([F_2])^\ast)$ whenever $\mathbb{P}(\zeta([F_1]\wedge [F_2])^\ast)=1$.

Suppose now that $[F_1],[F_2]$ are dependent and that $\mathbb{P}(\zeta([F_1]\wedge [F_2])^\ast)<1$. Likewise, it follows from Schemas (XI) and (VI) plus the contradual of Theorem I.4.20 that $\mathbb{P}(\zeta([F_1]\wedge [F_2])^\ast)<1$ if and only if ${\rm cl}~{\rm int}~\zeta([F_1]\wedge [F_2])^\ast\not=Z^\ast$ if and only if ${\rm cl}_{\zeta([F_2])^\ast}~{\rm int}_{\zeta([F_2])^\ast}~\zeta([F_1])^\ast\not=Z^\ast\cap\zeta([F_2])^\ast$ or ${\rm cl}~{\rm int}~\zeta([F_2])^\ast\not=Z^\ast$ if and only if $\zeta([F_1])^\ast\not=Z^\ast\cap\zeta([F_2])^\ast$ or $\zeta([F_2])^\ast\not=Z^\ast$ if and only if ${\rm cl}_{\zeta([F_2])^\ast}~{\rm int}_{\zeta([F_2])^\ast}~\zeta([F_1])^\ast\linebreak\not=Z^\ast\cap\zeta([F_2])^\ast$ or ${\rm cl}~{\rm int}~\zeta([F_2])^\ast\not=Z^\ast$ if and only if $\mathbb{P}_{\zeta([F_2])^\ast}(\zeta([F_1])^\ast)<1$ or $\mathbb{P}(\zeta([F_2])^\ast)<1$ if and only if $\mathbb{P}_{\zeta([F_2])^\ast}(\zeta([F_1])^\ast)\times\mathbb{P}(\zeta([F_2])^\ast)<1$ because $\mathbb{P}_{\zeta([F_2])^\ast}(\zeta([F_1])^\ast)\times\mathbb{P}(\zeta([F_2])^\ast)=1$ otherwise. Hence, since $\mathbb{P}(\zeta([F_1]\wedge [F_2])^\ast)={\bf p}([F_1]\wedge [F_2])={\bf p}_{[F_2]}([F_1])\times{\bf p}([F_2])=\mathbb{P}_{\zeta([F_2])^\ast}(\zeta([F_1])^\ast)\times\mathbb{P}(\zeta([F_2])^\ast)$ because $\zeta$ is a measure-preserving Boolean $\sigma$-isomorphism and $[F_1],[F_2]$ are dependent equivalence classes of formulas of $\mathcal{L}$, there exist semantic bridges such that $\mathbb{P}(\zeta([F_1]\wedge [F_2])^\ast)=\mathbb{P}_{\zeta([F_2])^\ast}(\zeta([F_1])^\ast)\times\mathbb{P}(\zeta([F_2])^\ast)$ whenever $\mathbb{P}(\zeta([F_1]\wedge [F_2])^\ast)\linebreak<1$. A glance at Definition 3.5 completes the proof.
\end{proof}

\begin{theorem}
{\it Let $(\boldsymbol{\mathfrak{A}},\boldsymbol{Z},\boldsymbol{\mathcal{E}},\zeta)$ be a model for the calculus of probability. Then there exists a semantic decomposition such that
\begin{gather*}
\hspace{-1cm}\textstyle{\mathbb{P}(\zeta([F_1]\wedge\ldots\wedge [F_n])^\ast)=\mathbb{P}(\zeta([F_1])^\ast)\times\mathbb{P}_{\zeta([F_1])^\ast}(\zeta([F_2])^\ast)\hspace{0.1cm}\times}\\
\hspace{1cm}\textstyle{\mathbb{P}_{\zeta([F_1]\,\wedge\,[F_2])^\ast}(\zeta([F_3])^\ast)\times\ldots\times\mathbb{P}_{\zeta([F_1]\,\wedge\,\ldots\,\wedge\,[F_{n-1}])^\ast}(\zeta([F_n])^\ast),}
\end{gather*}
the probability of the event $\zeta([F_1]\wedge\ldots\wedge [F_n])^\ast$ corresponding to the meet of $n$ not\pagebreak\linebreak necessarily mutually independent equivalence classes $[F_1],\ldots,[F_n]$ of formulas of $\mathcal{L}$, provided that the probability of $\zeta([F_1]\wedge\ldots\wedge [F_{n-1}])^\ast$ is positive, being the probability of the event $\zeta([F_1])^\ast$ corresponding to $[F_1]$ times the probability of the event $\zeta([F_2])^\ast$ corresponding to $[F_2]$ given the event $\zeta([F_1])^\ast$ corresponding to $[F_1]$, times the probability of the event $\zeta([F_3])^\ast$ corresponding to $[F_3]$ given the event $\zeta([F_1]\wedge [F_2])^\ast$ corresponding to the meet of $[F_1],[F_2]$, etc. times the probability of the event $\zeta([F_n])^\ast$ corresponding to $[F_n]$ given the event $\zeta([F_1]\wedge\ldots\wedge [F_{n-1}])^\ast$ corresponding to the meet of $[F_1],\ldots,[F_{n-1}]$.}
\end{theorem}

\begin{proof}
As an immediate consequence of Theorem 3.16.
\end{proof}

Theorems 3.16 and 3.17 establish the existence of a semantic decomposition for the product rule and the generalised product rule. Other properties having a semantic decomposition can be derived (e.g. the law of total probability or Bayes' theorem). These properties are not reproduced here.

\begin{theorem}
{\it Let $(\boldsymbol{\mathfrak{A}},\boldsymbol{Z},\boldsymbol{\mathcal{E}},\zeta)$ be a model for the calculus of probability. For every equivalence class $[F]$ of formulas of $\mathcal{L}$,
\begin{gather*}
\tag{XII}
\textstyle{\mathbb{P}(\zeta([F])^\ast\triangle\,B)=0~\text{if and only if}~{\rm int}~{\rm cl}~(\zeta([F])^\ast\triangle\,B)=\emptyset,}
\end{gather*}
the symmetric difference of the event $\zeta([F])^\ast$ corresponding to $[F]$ with a Baire set $B$ being negligible if and only if the symmetric difference of $\zeta([F])^\ast$ and $B$ is nowhere dense in $Z^\ast$ $($`impossible'\,$)$.}
\end{theorem}

\begin{proof}
As an immediate consequence of Theorem 2.5.
\end{proof}

\begin{theorem}
{\it Let $(\boldsymbol{\mathfrak{A}},\boldsymbol{Z},\boldsymbol{\mathcal{E}},\zeta)$ be a model for the calculus of probability. Then there exists a semantic decomposition such that
\begin{gather*}
\textstyle{\mathbb{P}(\zeta([F])^\ast\triangle\,B)=0~\text{if and only if}~B\in\zeta([F]),}
\end{gather*}
the symmetric difference of the event $\zeta([F])^\ast$ corresponding to an equivalence class $[F]$ of formulas of $\mathcal{L}$ with a Baire set $B$ being negligible if and only if $\zeta([F])^\ast$ and $B$ belong to the same equivalence class of Baire sets.}
\end{theorem}

\begin{proof}
For recall that, by Schema (VII) in \cite{Cosyns27}, for every equivalence class $[F]$ of formulas of $\mathcal{L}$, ${\rm int}~{\rm cl}~(\zeta([F])^\ast\triangle\,B)=\emptyset$ if and only if $B\in\zeta([F])$. Then $\mathbb{P}(\zeta([F])^\ast\triangle\,B)=0$ if and only if ${\rm int}~{\rm cl}~(\zeta([F])^\ast\triangle\,B)=\emptyset$ if and only if $B\in\zeta([F])$ by Schema (XII) and (VII) from which the conclusion follows.
\end{proof}

At first glance, the reader may contend that Theorem 3.19 is inconsequen-\linebreak tial: $\mathbb{P}(\zeta([F])^\ast\triangle\,B)$ is negligible, that is to say, $\zeta([F])^\ast$ and $B$ are {\it equal almost surely}, if and only if $\zeta([F])^\ast$ and $B$ belong to the same equivalence class of Baire sets, and the equivalence relation defined on $\mathcal{B}a(Z^\ast)$ naturally inducing a congruence relation on $\mathcal{B}a(Z^\ast)/\mathcal{N}$, from a measure-theoretic standpoint, they are equal in probability (i.e. $\mathbb{P}(\zeta([F])^\ast)=\mathbb{P}(B)$).

But the present state of affairs is more subtle. For recall that $\mathbb{P}(\zeta([F])^\ast)$, as a sentence of the metalanguage $\mathcal{L}^E$ of mathematical English, can always be decomposed into a statement of probability and a statement of modal logic by combining Schemas (VIII)--(XI) and (III)--(VI). Then $\mathbb{P}(\zeta([F])^\ast)$ is a theorem of $\mathcal{L}^E$ if and only if $\zeta([F])^\ast$ is open-closed because Schemas (III)--(VI) are theorems of $\mathcal{L}^E$ if and only if $\zeta([F])^\ast$ is open-closed \cite{Cosyns27}. Hence, although $\zeta([F])^\ast$ and $B$ are equal almost surely, they are not interchangeable {\it salva veritate}.

Theorem 3.19 uncovers, therefore, a probabilistic version of the concept of agnoia introduced in the first paper of this series \cite[Definition 4.22]{Cosyns27}: $\mathbb{P}(\zeta([F])^\ast)$, as a sentence of $\mathcal{L}^E$, can be false in the sense that its semantic decomposition is not a theorem of $\mathcal{L}^E$, in which case $\mathbb{P}(\zeta([F])^\ast)$, as a probability measure, is meaningless. For example, consider the semantic decomposition
\begin{gather*}
\textstyle{\mathbb{P}(\zeta([F])^\ast)=0~\text{if and only if}~{\rm int}~{\rm cl}~\zeta([F])^\ast=\emptyset}\\
\textstyle{\text{if and only if}~\zeta([F])^\ast=\emptyset,}
\end{gather*}
and suppose that $\zeta([F])^\ast$ is closed. If $\zeta([F])^\ast$ is the empty set, then it is certainly nowhere dense in $Z^\ast$, hence negligible, whilst if $\zeta([F])^\ast$ is negligible, then it is nowhere dense in $Z^\ast$, although not necessarily the empty set, the homeomorphic copy of the Cantor set in $Z^\ast$ being a classic example.

\begin{definition}
Let $\mathbb{P}$ and ${\rm P}$ be probability measures on the $\sigma$-field $\mathcal{B}a(Z^\ast)$ of Baire subsets of $Z^\ast$. Then\\
(i) ${\rm P}$ is {\it absolutely continuous with respect to $\mathbb{P}$} if, for every $B\in\mathcal{B}a(Z^\ast)$, ${\rm P}(B)\linebreak=0$ whenever $\mathbb{P}(B)=0$;\\
(ii) ${\rm P}$ and $\mathbb{P}$ are {\it singular with respect to each other}, or {\it mutually singular}, if there exists a $B\in\mathcal{B}a(Z^\ast)$ such that ${\rm P}(B)=\mathbb{P}(Z^\ast\setminus B)=0$.
\end{definition}

\begin{theorem}
{\it Let $(\boldsymbol{\mathfrak{A}},\boldsymbol{Z},\boldsymbol{\mathcal{E}},\zeta)$ be a model for the calculus of probability. Then there exist no finer measures of agnoia than $\mathbb{P}$, that is to say, if ${\rm P}$ is a strictly positive, countably additive, normal probability measure on the $\sigma$-field $\mathcal{B}a(Z^\ast)$ of Baire subsets of the Stone space $Z^\ast$ of a measurable Lindenbaum--Tarski algebra $\mathfrak{A}^\ast$, then ${\rm P}$ is absolutely continuous with respect to $\mathbb{P}$.}
\end{theorem}

\begin{proof}
Since ${\rm P}$ and $\mathbb{P}$ are probability measures, there exists a Lebesgue decomposition \cite[Theorem 8.11]{Bartle95} ${\rm P}={\rm P}_1+{\rm P}_2$ such that ${\rm P}_1$ is absolutely continuous with respect to $\mathbb{P}$ and ${\rm P}_2$ is singular with respect to $\mathbb{P}$. In particular, if ${\rm P}_2$ and $\mathbb{P}$ are mutually singular, then there exists a $B\in\mathcal{B}a(Z^\ast)$ such that ${\rm P}_2(B)=\mathbb{P}(Z^\ast\setminus B)=0$. A contradiction because every Baire subset of $Z^\ast$ is negligible if and only if it is nowhere dense and, in every topological space, the complement of a nowhere dense set is non-boundary. It follows that no singular, normal probability measure ${\rm P}_2$ with respect to $\mathbb{P}$ exists on $\mathcal{B}a(Z^\ast)$, whence\linebreak ${\rm P}={\rm P}_1$ is absolutely continuous with respect to $\mathbb{P}$ as required.
\end{proof}

Theorem 3.21 shows that agnoia is irremediable up to normal probability\linebreak measures. The systematic study of non-normal probability measures in hyperstonean spaces is a subdomain of formal agnoiology.


\section{Conclusion}

This study presented a solution to the problem of the efficient integration of logic and probability. As in the first paper of this series, the question naturally arises as to how valuable this contribution to the literature is. The following initial appraisal can be made:\\
(i) The proposed solution requires that logic, algebra, topology, and measure theory be tightly fitted together, excluding in effect all probability spaces which are not hyperstonean or do not have a hyperstonean cover.

There are, however, several advantages in adopting this integration of logic and probability over that of other proposals. In particular:\\
(ii) The proposed solution leaves the axiomatic treatment of probability formulated by Kolmogorov unscathed by bringing to the fore the inner logic of the calculus of probability instead of assigning probabilities to sentences;\\
(iii) The proposed solution addresses the difficulties surrounding the interpretation of events whose symmetric difference is negligible. In fact, these events are exactly those giving rise to a probabilistic version of agnoia;\\
(iv) The proposed solution delineates a subdomain of formal agnoiology as the study of non-normal probability measures in hyperstonean spaces, potentially paving the way for the discovery of interesting results in uncertainty modelling, especially in relation to von Neumann algebras.

This last aspect of modal probability spaces has only been touched upon in this study and is part of a larger inquiry into formal agnoiology, which is the primary subject matter of the third paper of this series on the foundations of logic and probability.


\end{document}